\documentclass[a4paper,leqno,10pt]{amsart}

\usepackage{epsf,graphicx,epsfig}
\usepackage{amscd}
\usepackage{amsmath,latexsym,amssymb,amsthm}
\usepackage[nospace,noadjust]{cite}
\usepackage{textcomp}
\usepackage{setspace,cite}
\usepackage{lscape,fancyhdr,fancybox}
\usepackage{stmaryrd}
\usepackage[all,cmtip]{xy}
\usepackage{tikz}
\usepackage{cancel}
\usetikzlibrary{shapes,arrows,decorations.markings}
\usepackage{graphicx}

\usepackage{color}
\usepackage{url}
\usepackage{enumerate}
\usepackage[mathscr]{euscript}
  \theoremstyle{plain}

\swapnumbers
    \newtheorem{thm}{Theorem}[section]
    \newtheorem{prop}[thm]{Proposition}
     
   \newtheorem{lemma}[thm]{Lemma}

    \newtheorem{subsec}[thm]{}
\theoremstyle{definition}
    \newtheorem{defn}[thm]{Definition}
        \newtheorem{remark}[thm]{Remark}
    \newtheorem{exam}[thm]{Example}

\theoremstyle{remark}

\title{}
\author{}
\date{}
\usepackage{amssymb}

\usepackage{hyperref}
\hypersetup{
	colorlinks,
	citecolor=blue,
	filecolor=black,
	linkcolor=blue,
	urlcolor=black
}

\begin{document}
\title{Gerstenhaber algebra structure on the cohomology of a hom-associative algebra}

\author{Apurba Das}

\curraddr{Department of Mathematics and Statistics, Indian Institute of Technology, Kanpur 208016, Uttar Pradesh, India}
\email{apurbadas348@gmail.com}


\begin{abstract}
A hom-associative algebra is an algebra whose associativity is twisted by an algebra homomorphism. In this paper, we define a cup product on the cohomology of a hom-associative algebra. A direct verification shows that this cup product
together with the degree $-1$ graded Lie bracket (which controls the deformation of the hom-associative algebra structure) on the cohomology makes it a Gerstenhaber algebra.
\end{abstract}

\noindent

\thispagestyle{empty}

\maketitle

\vspace{0.2cm}
\section{Introduction}
It is a classical fact that the Hochschild cohomology $H^\bullet(A, A)$ of an associative algebra
$A$ carries a Gerstenhaber algebra structure \cite{gers}. A (right) Gerstenhaber algebra is a graded commutative, associative algebra $(\mathcal{A} = \oplus_{i \in \mathbb{Z}} \mathcal{A}^i, \cup)$ together with a degree $-1$ graded Lie bracket $[-,-]$ compatible with the product $\cup$ in the sense of the following Leibniz rule
\begin{align}\label{g-compatibility}
[a \cup b, c] = [a, c] \cup b + (-1)^{(|c|-1) |a|} a \cup [b, c].
\end{align}
For an associative algebra $(A, \mu)$, the Hochschild cochain groups $C^\bullet (A, A)$ carry a (cup) product $\cup : C^m (A, A) \otimes C^n(A, A) \rightarrow C^{m+n} (A, A)$ defined by
\begin{align}\label{std-cup}
(f \cup g) (a_1, \ldots, a_{m+n}) = \mu ( f (a_1, \ldots, a_m) , g(a_{m+1}, \ldots, a_{m+n}) ).
\end{align}
It turns out that the Hochschild coboundary operator $\delta$ is a graded derivation with respect to the cup product. Hence, it induces a cup product $\cup$ on the Hochschild cohomology $H^\bullet (A, A)$. Moreover, the cochain groups $C^\bullet (A, A)$ carry a degree $-1$ graded Lie bracket compatible with the Hochschild cobundary $\delta$ \cite{gers}. Therefore, it gives rise to a degree $-1$ graded Lie bracket on $H^\bullet (A, A)$. The cup product and the degree $-1$ graded Lie bracket on the Hochschild cohomology $H^\bullet (A, A)$ are compatible in the sense of (\ref{g-compatibility}) to make it into a Gerstenhaber algebra.

Recently, hom type algebras have been studied by many authors. In these algebras, the identities defining the structures are twisted by linear maps.
The notion of hom-Lie algebras was introduced by Hartwig, Larsson and Silvestrov \cite{hls}.
 Hom-Lie algebras appeared in examples of $q$-deformations of the Witt and Virasoro algebras. Other type of algebras (e.g. associative, Leibniz, Poisson, Hopf,...) twisted by linear maps have also been studied. See \cite{makh-sil, makh-sil2} (and references there in) for more details. Our main object of interest in this paper is the notion of hom-associative algebra introduced by Makhlouf and Silvestrov \cite{makh-sil}. A hom-associative algebra is an algebra $(A, \mu)$ whose associativity is twisted by a linear homomorphism $\alpha : A \rightarrow A$. It is called multiplicative if $\alpha$ is an algebra map. The commutator bracket of a hom-associative algebra gives rise to a hom-Lie algebra. More recently, a step towards the noncommutative geometry of hom-associative algebras have been initiated in \cite{hss}.

In \cite{amm-ej-makh, makh-sil2} the authors studied the formal one-parameter deformation of multiplicative hom-associative algebras and introduce a Hochschild type cohomology theory for them. Given a multiplicative hom-associative algebra $(A, \mu, \alpha)$, its $n$-th cochain group $C^n_\alpha (A, A)$ consists of multilinear maps $f : A^{\otimes n} \rightarrow A$ which satisfies $ \alpha \circ f = f \circ \alpha^{\otimes n} $ and the coboundary operator $\delta_\alpha : C^n_\alpha (A, A) \rightarrow C^{n+1}_\alpha (A, A)$ is similar to the Hochschild coboundary but suitably twisted by $\alpha$. We denote the cohomology by $H^\bullet_\alpha (A, A)$. Cohomology with coefficients in a suitable bimodule has been studied in \cite{arfa-fraj-makh}. We show that the second Hochschild cohomology group with coefficients in a bimodule are in one-to-one correspondence with abelian extensions (cf. Theorem \ref{thm-abelian-ext}). In our knowledge, this result has not yet been mentioned in any article.

Next, we turn our attention to the Hochschild cohomology of a multiplicative hom-associative algebra $(A, \mu, \alpha)$ with coefficients in itself. Inspired from the classical case, the authors in \cite{amm-ej-makh} introduce a degree $-1$ graded Lie bracket $[-,-]_\alpha$ on the cochain groups $C^\bullet_\alpha (A, A)$. However, they did not study further on the structure of the cohomology.
In this paper, we show that the bracket $[-,-]_\alpha$ on $C^\bullet_\alpha (A, A)$ induces a degree $-1$ graded Lie bracket on the cohomology $H^\bullet_\alpha (A, A)$. Next, we introduce a cup product $\cup_\alpha$ on the cochain groups $C^\bullet_\alpha (A, A)$. Although, the algebra $A$ is not associative, the cup product $\cup_\alpha$ turns out to be associative on $C^\bullet_\alpha (A, A)$. Moreover, the coboundary operator $\delta_\alpha$ is a graded derivation with respect to $\cup_\alpha$. Therefore, it induces an associative cup product (also denoted by $\cup_\alpha$) on the cohomology $H^\bullet_\alpha (A, A)$. The cup product on the cochain level may not be graded commutative but it is so at the level of cohomology. Finally, the cup product $\cup_\alpha$ and the degree $-1$ graded Lie bracket $[-,-]_\alpha$ on the cohomology $H^\bullet_\alpha (A, A)$ satisfy the Leibniz rule (\ref{g-compatibility}) to become a Gerstenhaber algebra (cf. Theorem \ref{final-thm-old}).

Throughout this paper all vector spaces are over a field $\mathbb{K}$ of characteristic $0$.

\section{Hom-associative algebras and graded pre-Lie algebras}
The aim of this section is to recall some preliminaries on multiplicative hom-associative algebras, its Hochschild type cohomology \cite{amm-ej-makh, arfa-fraj-makh, makh-sil, makh-sil2} and graded pre-Lie algebras \cite{gers}.
\begin{defn}
A hom-associative algebra is a triple $(A, \mu, \alpha)$ consisting of a vector space $A$ together with a bilinear map
$\mu : A \times A \rightarrow A$ and a linear map $\alpha : A \rightarrow A$ satisfying
\begin{align}\label{hom-ass-cond}
\mu ( \alpha (a) , \mu ( b , c) ) = \mu ( \mu (a , b) , \alpha (c)), ~~ \text{ for all } a, b, c \in A.
 \end{align}
A hom-associative algebra $(A, \mu, \alpha)$ is called multiplicative if
$$ \alpha ( \mu( a, b)) = \mu (\alpha (a) , \alpha (b)), ~~ \text{ for all } a, b \in A.$$
\end{defn}

In the rest of the paper, by a hom-associative algebra, we shall always mean a multiplicative hom-associative algebra. When $\alpha =$ identity, in any case, one gets the definition of a classical associative algebra.

\begin{exam}
Let $(A, \mu)$ be an associative algebra and $\alpha : A \rightarrow A$ be an algebra homomorphism. Then the triple $(A, \alpha \circ \mu, \alpha)$ is a hom-associative algebra, called `induced by composition'.
\end{exam}

\begin{exam}
Let $A = (A, \mu, \alpha)$ and $A' = (A' , \mu', \alpha')$ be two hom-associative algebras. Then $A \otimes A' = (A \otimes A', \mu \otimes \mu', \alpha \otimes \alpha')$ is a hom-associative algebra, called the tensor product of $A$ and $A'$.
\end{exam}

\begin{exam}
Let $ (A, \mu, \alpha)$ be a hom-associative algebra. Then $(M_n (A), M_n (\mu), M_n (\alpha))$ is also a hom-associative algebra whose multiplication $M_n (\mu)$ is given by the matrix multiplication induced by $\mu$ and the structure map $M_n (\alpha)$ is the map $\alpha$ on each entry of the matrix.
\end{exam}

\begin{exam}
Let $A$ be a two dimensional vector space with basis $\{e_1, e_2\}$. Define a multiplication $\mu : A \times A \rightarrow A$ by
$$ \mu (e_i, e_j) = \begin{cases} e_1 & \mbox{ if } (i,j) = (1,1)\\
e_2 & \mbox{ if } (i,j) \neq (1,1).
\end{cases} $$
A linear map $\alpha : A \rightarrow A$ is defined by $\alpha (e_1) = e_1 - e_2$ and $\alpha (e_2) = 0$. Then the triple $(A, \mu, \alpha)$ forms a hom-associative algebra.
\end{exam}

Next, we recall the definition of Hochschild type cohomology for hom-associative algebras \cite{amm-ej-makh}. We first recall bimodules over hom-associative algebras.

\begin{defn} Let $A= (A, \mu, \alpha)$ be a hom-associative algebra. A left $A$-module
 consists of a triple $M = (M, \mu_l, \alpha_M)$ of a vector space $M$, a bilinear map $\mu_l : A \otimes M \rightarrow M$ and a linear  map $\alpha_M : M \rightarrow M$ satisfying
\begin{align*}
\alpha_M ( \mu_l(a , m)) = \mu_l (\alpha (a) , \alpha_M (m)) \text{ and } ~~~ \mu_l (\alpha (a) , \mu_l (b , m)) = \mu_l ( \mu (a, b) , \alpha_M (m)).
\end{align*}
\end{defn}
A right $A$-module $M = (M, \mu_r, \alpha_M)$ can be defined in a similar way. An $A$-bimodule is a quadruple $M = (M, \mu_l, \mu_r, \alpha_M)$ so that $(M, \mu_l, \alpha_M)$ is a left $A$-module and $(M, \mu_r, \alpha_M)$ is a right $A$-module and further satisfy
\begin{align*}
\mu_l (\alpha(a) , \mu_r (m , b)) = \mu_r (\mu_l (a , m) , \alpha (b)).
\end{align*}

\begin{remark}\label{self-bimod}
Any hom-associative algebra $A$ is a bimodule over itself with $\mu_l = \mu_r = \mu$ and $\alpha_M = \alpha$. When there is an $A$-bimodule $M = (M, \mu_l, \mu_r, \alpha_M)$, we denote both $\mu_l$ and $\mu_r$ by $\mu$ if there is no confusion and simply say $M = (M, \mu, \alpha_M)$ is an $A$-bimodule. Therefore, by $\mu$ we denote the product in $A$ as well as the left and the right action of $A$ on $M$. It would be clear from the entries of $\mu$.
\end{remark}
%Sometimes we will also use $\cdot$ (dot) to denote either of the left or right action of $A$ on $M$.

If $M = (M, \mu, \alpha_M)$ is an $A$-bimodule, one can define a hom-associative algebra structure on $M \oplus A$. The product and the structure map are given by
\begin{align*}
\overline{\mu} ((m, a), (n, b)) =~ ( \mu(m , b) + \mu (a , n),~ \mu (a, b)) ~~~\text{~~~ and ~~~}~~~ \overline{\alpha}((m, a)) = (\alpha_M (a), \alpha (a)).
\end{align*}
This is called the semi-direct product hom-associative algebra.

\medskip

Let $(A, \mu, \alpha)$ be a hom-associative algebra and $M$ be an $A$-bimodule. For each $n \geq 1$, we define a $\mathbb{K}$-vector space $C^n_{\alpha, \alpha_M} (A,M) $ consisting of all multilinear maps $f : A^{\otimes n} \rightarrow M$ satisfying $\alpha_M \circ f = f \circ \alpha^{\otimes n}$, that is,
$$ (\alpha_M \circ f) (a_1, \ldots, a_n) =  f ( \alpha(a_1), \ldots , \alpha (a_n )), ~~\text{ for all }a_i \in A .$$
Define $\delta_\alpha : C^n_{\alpha, \alpha_M} (A, M) \rightarrow C^{n+1}_{\alpha, \alpha_M} (A, M)$ by the following
\begin{align*}
 (\delta_\alpha f) (a_1, a_2, \ldots, a_{n+1}) =~& \mu \big(\alpha^{n-1}(a_1) , f(a_2, \ldots , a_{n+1}) \big) \\
~& + \sum_{i=1}^{n} (-1)^i f \big( \alpha(a_1), \ldots, \alpha (a_{i-1}), \mu (a_i , a_{i+1}), \alpha (a_{i+2}), \ldots, \alpha (a_{n+1}) \big) \\
~& + (-1)^{n+1} \mu (f (a_1, \ldots, a_n) , \alpha^{n-1} (a_{n+1})).
\end{align*}
It has been shown in \cite{amm-ej-makh} that $\delta^2_\alpha = 0$. Therefore, if $Z^n_{\alpha, \alpha_M} (A,M) = \{ f \in C^n_{\alpha, \alpha_M} (A,M)|~ \delta_\alpha (f) = 0 \}$ denote the space of $n$-cocycles and $B^n_{\alpha, \alpha_M} (A,M) = \{ \delta_\alpha g |~ g \in C^{n-1}_{\alpha, \alpha_M} (A,M) \}$ denote the space of $n$-coboundaries, we have $ B^n_{\alpha, \alpha_M} (A,M) \subseteq Z^n_{\alpha, \alpha_M} (A,M)$. The $n$-th cohomology group of the hom-associative algebra $(A, \mu , \alpha)$ with coefficients in $M$ is defined by
$$ H^n_{\alpha, \alpha_M} (A, M) := \frac{Z^n_{\alpha, \alpha_M} (A,M)}{B^n_{\alpha, \alpha_M} (A,M)  }, ~~ \text{ for } n \geq 2.$$ 

When $M = A$ with the $A$-bimodule structure given in Remark \ref{self-bimod}, we denote the corresponding cochain groups simply by $C^n_\alpha (A, A)$ and the cohomology by $H^n_\alpha (A, A),$ for $n \geq 2$.

\begin{remark}
(i) In general, the cohomology groups of hom-associative algebras does not make sense for $n= 0, 1$. In the case of associative algebra and bimodule over it (i.e. $\alpha = $ id, $\alpha_M =$ id), one can define the space of zero-th cochains by $M$ and the coboundary map by $(\delta m)(a) = \mu (a, m) - \mu (m, a)$, for $m \in M,~ a \in A$. Thus, the cohomology groups make sense for any $n \geq 0$ and one recovers the classical Hochschild cohomology of an associative algebra.

(ii) Let $(A, \mu, \alpha)$ be a hom-associative algebra. Then $\mu \in C^2_\alpha (A, A)$. It follows from the hom-associativity condition that  $\delta_\alpha (\mu) = 0$. Therefore, $\mu$ defines a $2$-cocycle. This is in fact a coboundary and is given by $\mu = \delta_\alpha (\text{id})$ where $\text{id} \in C^1_\alpha (A, A)$ is the identity map on $A$. 
\end{remark}

Next, we recall some details on graded pre-Lie algebras \cite{gers}.

\begin{defn}
A pre-Lie system $\{ V_m, \circ_i \}$ consists of a sequence of vector spaces $V_m, m \geq 0$, and for each $m , n \geq 0$ there exist linear maps $$ \circ_i : V_m \otimes V_n \rightarrow V_{m+n}, ~~ (f,g) \mapsto f \circ_i g ~~~ \qquad 0 \leq i \leq m$$
satisfying
$$ ( f \circ_i g) \circ_j h = \begin{cases} (f \circ_j h ) \circ_{i+p} g & \mbox{ if } \quad 0 \leq j \leq i-1\\
f \circ_i ( g \circ_{j-i} h ) & \mbox{ if } \quad i \leq j \leq n+1,
\end{cases} $$
for $f \in V_m,~ g \in V_n$ and $h \in V_p$.
\end{defn}

\begin{exam}
Let $A$ be a vector space and define $V_m$ to be the space of all multilinear maps from $A^{\otimes m+1}$ to $A$ for $m \geq 0$.  If $f \in V_m$ and $g \in V_n$ with $m, n \geq 0$, then for $0 \leq i \leq m$, define $f \circ_i g \in V_{m+n}$ by
$$(f \circ_i g ) (a_0, a_1, \ldots, a_{m+n}) = f \big(a_0, \ldots, a_{i-1}, g (a_i, a_{i+1}, \ldots, a_{i+n}), a_{i+n+1}, \ldots, a_{m+n} \big).$$
%One may also extend the product when $g \in V_{-1} = A$ by the following
%$$ (f \circ_i g) (a_0, a_1, \ldots, a_{m-1}) = f (a_0, \ldots, a_{i-1}, g, a_i, \ldots, a_{m-1}).$$
In the next section, we generalize this example to a vector space $A$ equipped with a linear map $\alpha$. 
\end{exam}

Let $\{ V_m , \circ_i\}$ be a pre-Lie system. Then for any $m, n \geq 0$, there is a new linear map $$\circ : V_m \otimes V_n \rightarrow V_{m+n},~~~ (f,g) \mapsto f \circ g$$ defined by
$$ f \circ g = \sum_{i=0}^{m} (-1)^{ni}  f \circ_i g , \text{ for } f \in V_m, ~  g \in V_n. $$

We will need the following useful result \cite[Theorem 2]{gers}.
\begin{thm}\label{thm2-gers}
Let $\{ V_m, \circ_i \}$ be a pre-Lie system and $f \in V_m,~ g \in V_n$ and $h \in V_p$, respectively. Then
\begin{itemize}
	\item[(i)] $(f \circ g) \circ h - f \circ (g \circ h) = \sum_{i, j} (-1)^{ni + pj} (f \circ_i g) \circ_j h$, where the summation is indexed over those $i$ and $j$ with either $0 \leq j \leq i-1$ or $n+i +1 \leq j \leq m+n$,
	\item[(ii)]  $  (f \circ g) \circ h - f \circ (g \circ h)  =  (-1)^{np} [(f \circ h) \circ g  -  f \circ (h \circ g) ].$
\end{itemize}
\end{thm}

\begin{defn}
A graded pre-Lie algebra $\{ V_m, \circ \}$ consists of vector spaces $V_m$ together with linear maps $\circ : V_m \otimes V_n \rightarrow V_{m+n}$ satisfying
$$  (f \circ g) \circ h - f \circ (g \circ h)  =  (-1)^{np} [(f \circ h) \circ g  -  f \circ (h \circ g) ], \text{ for } f \in V_m,~ g \in V_n,~ h \in V_p.$$
\end{defn}

It is proved in \cite{gers} that the graded commutator of a graded pre-Lie algebra defines a graded Lie algebra structure. More precisely, if 
 $\{ V_m , \circ \}$ is a graded pre-Lie algebra, then the bracket $[-,-]$ defined by
 $$ [f, g] = f \circ g - (-1)^{mn} g \circ f, ~~ \text{ for } f \in V_m,~ g \in V_n$$
 defines a graded Lie algebra structure on $\oplus {V_m}.$
 
 \begin{remark}\label{pre-sys-pre-lie}
 	It follows from Theorem \ref{thm2-gers}(ii) that if $\{ V_m, \circ_i \}$ is a pre-Lie system then $\{ V_m, \circ\}$ forms a graded pre-Lie algebra. Therefore, by applying the graded commutator we get a graded Lie algebra structure on $\oplus V_m$.
 \end{remark}

\section{Abelian extensions} In this section, we show that the second Hochschild cohomology $H^2_{\alpha, \alpha_M} (A, M)$ of a hom-associative algebra $A$ with coefficients in a bimodule $M$ can be interpreted as equivalence classes of abelian extensions of $A$ by $M$.

Let $A = (A, \mu, \alpha)$ be a hom-associative algebra and $M= (M, \alpha_M)$ be a vector space equipped with a linear map $\alpha_M : M \rightarrow M$. Note that $M$ can be considered as a hom-associative algebra with trivial multiplication.

\begin{defn}
An abelian extension of $A$ by $M$ is an exact sequence of hom-associative algebras
%\begin{align*}
\[
\xymatrix{
0 \ar[r] &  (M, 0, \alpha_M) \ar[r]^{i} & (E, \mu_E, \alpha_E) \ar[r]^{j} & (A, \mu, \alpha) \ar[r] \ar@<+4pt>[l]^{s} & 0
}
\]
together with a $\mathbb{K}$-splitting (given by $s$) which satisfies 
\begin{align}\label{s-property}
\alpha_E \circ s = s \circ \alpha.
\end{align}
%\end{align*}
\end{defn}

An abelian extension induces an $A$-bimodule structure on $(M, \alpha_M)$ via the action map $\mu (a , m) := \mu_E (s(a), i(m))$ and $\mu (m , a) := \mu_E (i(m), s(a))$, for $a \in A,~ m \in M.$ One can easily verify that this action in independent of the choice of $s$. 

\begin{remark}
Let $(E, \alpha_E)$ and $(A, \alpha)$ be two vector spaces equipped with linear maps $\alpha_E : E \rightarrow E$ and $\alpha : A \rightarrow A$. Suppose $j : E \rightarrow A$ is a linear surjective map which commutes with respective structure maps. Then there might not be a section $s : A \rightarrow E$ of $j$ which commute with respective structure maps. Take $E = \langle x , y \rangle$ with $\alpha_E (x) = y,~ \alpha_E (y) = 0$ and $A = \langle a \rangle$ with $\alpha (a) = 0.$ Take $j (x) = a$ and $j (y) = 0$. Let $s$ be a section for $j$ commuting with respective structure maps. For $s (a) = \lambda x + \nu y$, we have $a = (j \circ s) (a) = \lambda a$, which implies that $\lambda = 1$. Finally,
\begin{align*}
0 = (s \circ \alpha) (a) = (\alpha_E \circ s)(a) = \alpha_E (x + \nu y) = y
\end{align*}
which is a contradiction.
\end{remark}

Two abelian extensions are said to be equivalent if there is a morphism $\phi : E \rightarrow E'$ between hom-associative algebras making the following diagram commute
\[
\xymatrix{
0 \ar[r] &  (M, 0, \alpha_M) \ar[r]^{i} \ar@{=}[d] & (E, \mu_E, \alpha_E) \ar[d]^{\phi} \ar[r]^{j} & (A, \mu, \alpha) \ar[r] \ar@{=}[d] \ar@<+4pt>[l]^{s} & 0 \\
0 \ar[r] &  (M, 0, \alpha_M) \ar[r]^{i'} & (E', \mu'_E, \alpha'_E) \ar[r]^{j'} & (A, \mu, \alpha) \ar[r] \ar@<+4pt>[l]^{s'} & 0 .
}
\]
Note that two extensions with same $i$ and $j$ but different $s$ are always equivalent.

Suppose $M$ is a given $A$-bimodule. We denote by $\mathcal{E}xt (A, M)$ the equivalence classes of abelian extensions of $A$ by $M$ for which the induced $A$-bimodule structure on $M$ is the pescribed one.

The next result is inspired from the classical case.
\begin{thm}\label{thm-abelian-ext}
$H^2_{\alpha, \alpha_M} (A, M) \cong \mathcal{E}xt (A, M).$
\end{thm}

\begin{proof}
Given a $2$-cocycle $f \in C^2_{\alpha, \alpha_M} (A, M)$, we consider the $\mathbb{K}$-module $E = M \oplus A$ with following structure maps
\begin{align*}
{\mu}_E ((m, a), (n, b)) =~& ( \mu(m , b) + \mu( a , n) + f (a, b),~ \mu (a, b)),\\
{\alpha}_E((m, a)) =~& (\alpha_M (a), \alpha (a)).
\end{align*}
(Observe that when $f =0$ this is the semi-direct product.)
Using the fact that $f$ is a $2$-cocycle, it is easy to verify that $(E, \mu_E, \alpha_E)$ is a hom-associative algebra. Moreover, $0 \rightarrow M \rightarrow E \rightarrow A \rightarrow 0$ defines an abelian extension with the obvious splitting. Let $(E' = M \oplus A, \mu_E', \alpha_E)$ be the corresponding hom-associative algebra associated to the cohomologous $2$-cocycle $f - \delta_{\alpha} (g)$, for some $g \in C^1_{\alpha, \alpha_M} (A, M)$. The equivalence between abelian extensions $E$ and $E'$ is given by $E \rightarrow E'$, $(m, a) \mapsto (m + g (a), a)$. Therefore, the map $H^2_{\alpha, \alpha_M} (A, M) \rightarrow \mathcal{E}xt (A, M) $ is well defined.

Conversely, given an extension 
$0 \rightarrow M \xrightarrow{i} E \xrightarrow{j} A \rightarrow 0$ with splitting $s$, we may consider $E = M \oplus A$ and $s$ is the map $s (a) = (0, a).$ With respect to the above splitting, the maps $i$ and $j$ are the obvious ones. Moreover, the property (\ref{s-property}) then implies that $\alpha_E = (\alpha_M, \alpha)$. Since $j \circ \mu_E ((0, a), (0, b)) = \mu (a, b)$ as $j$ is an algebra map, we have $\mu_E ((0, a), (0, b)) = (f (a, b), \mu (a, b))$, for some $f \in C^2_{\alpha, \alpha_M} (A, M).$ The hom-associativity of $\mu_E$ then implies that $f$ is a $2$-cocycle. Similarly, one can observe that any two equivalent extensions are related by a map $E = M \oplus A \xrightarrow{\phi} M \oplus A = E'$, $(m, a) \mapsto (m + g(a), a)$ for some $g \in C^1_{\alpha, \alpha_M} (A, M)$. Since $\phi$ is an algebra morphism, we have
\begin{align*}
\phi \circ \mu_E ((0, a), (0, b)) = \mu'_{E} (\phi (0, a) , \phi (0, b))
\end{align*}
which implies that $f' (a, b) = f (a, b) - (\delta_{\alpha} g)(a, b)$. Here $f'$ is the $2$-cocycle induced from the extension $E'$. This shows that the map $\mathcal{E}xt (A, M) \rightarrow H^2_{\alpha, \alpha_M} (A, M)$ is well defined. Moreover, these two maps are inverses to each other.
%$\xymatrix{
%0 \ar[r] & M \ar[r]^{i} & E \ar[r]^{j} & A \ar[r] & 0
%}$
\end{proof}

In the following sections, we will be interested in the Hochschild cohomology of $A$ with coefficients in itself.

\section{Shifted graded Lie bracket}\label{sec-3}

Let $A$ be a vector space and $\alpha : A \rightarrow A$ be a linear map. For each $m \geq 0$, define $V_m^\alpha$ to be the space of all multilinear maps $f : A^{\otimes m+1} \rightarrow A$ satisfying 
$$ (\alpha \circ f) (a_1, \ldots, a_{m+1}) =  f ( \alpha(a_1), \ldots , \alpha (a_{m+1} )), ~~\text{ for all }a_i \in A .$$

It is proved in \cite[Theorem 3.3]{amm-ej-makh} that the graded vector space $\oplus_{m \geq 0} V_m$ carries a graded Lie algebra structure. Here, we will give an alternative proof of the same fact. Our construction of the same graded Lie algebra structure is based on the pre-Lie system.

If $f \in V_m^\alpha$ and $g \in V_n^\alpha$ with $m, n \geq 0$, then for $0 \leq i \leq m$, define $f \circ_i g \in V_{m+n}^\alpha$ by
$$(f \circ_i g ) (a_0, a_1, \ldots, a_{m+n}) = f \big( \alpha^n a_0, \ldots, \alpha^n a_{i-1}, g (a_i, a_{i+1}, \ldots, a_{i+n}), \alpha^n a_{i+n+1}, \ldots, \alpha^n a_{m+n} \big).$$

In the rest of the paper, by $\circ_i$ operation, we shall always mean the operation defined above (involving $\alpha$).

\begin{prop}\label{new-sys}
	With the above notations $\{ V_m^\alpha , \circ_i\}$ forms a pre-Lie system.
\end{prop}
\begin{proof}
	Let $f \in V_m^\alpha,~ g \in V_n^\alpha$ and $h \in V_p^\alpha$.
	For $0 \leq j \leq i-1$ we have
	\begin{align*}
&~((f \circ_j h) \circ_{i+p} g ) (a_0, a_1, \ldots, a_{m+n+p})\\
=&~ (f \circ_j h) \big(  \alpha^n a_0, \ldots, \alpha^n a_{i+p-1} , g \big( a_{i+p}, \ldots, a_{i+p+n}  \big), \alpha^n a_{i+p+n+1}, \ldots, \alpha^n a_{m+n+p}  \big) \\
=&~ f \big( \alpha^{n+p} a_0, \ldots, \alpha^{n+p} a_{j-1}, h \big( \alpha^n a_j, \ldots, \alpha^n a_{j+p}  \big), \ldots, \alpha^{n+p} a_{i+p-1}, \alpha^p  g \big( a_{i+p}, \ldots, a_{i+p+n}  \big),   \ldots, \alpha^{n+p} a_{m+n+p}    \big) \\
=&~ (f \circ_i g) \big( \alpha^p a_0, \ldots, \alpha^p a_{j-1}, h \big(a_j, \ldots, a_{j+p} \big), \alpha^p a_{j + p +1}, \ldots, \alpha^{p} a_{m+n+p}     \big) \\
=&~ (  (f \circ_i g) \circ_j h) (a_0, a_1, \ldots, a_{m+n+p}).
\end{align*}
Similarly, for $1 \leq j \leq n+1$ we have
\begin{align*}
&~(f \circ_i (g \circ_{j-i} h)) (a_0, a_1, \ldots, a_{m+n+p}) \\
=&~ f \big( \alpha^{n+p} a_0, \ldots, \alpha^{n+p} a_{i-1},  (g \circ_{j-i} h) (a_i, \ldots, a_{i+n+p}), \alpha^{n+p} a_{i+n+p+1}, \ldots, \alpha^{n+p} a_{m+n+p} \big) \\
=&~ f \big( \alpha^{n+p} a_0, \ldots, \alpha^{n+p} a_{i-1}, g \big( \alpha^p a_i, \ldots, \alpha^p a_{j-1}, h (a_j, \ldots, a_{j+p}), \ldots, \alpha^p a_{i+n+p} \big), \ldots, \alpha^{n+p} a_{m+n+p} \big) \\
=&~ (f \circ_i g) \big( \alpha^p a_0, \ldots, \alpha^p a_{i-1}, \alpha^p a_i, \ldots, \alpha^p a_{j-1}, h \big(  a_j, \ldots, a_{j+p}    \big), \alpha^p a_{j+p+1} , \ldots, \alpha^p a_{m+n+p}  \big) \\
=&~ ((f \circ_i g) \circ_j h ) (a_0, a_1, \ldots, a_{m+n+p}).
\end{align*}
\end{proof}

Therefore, it follows from Remark \ref{pre-sys-pre-lie} that given a vector space $A$ together with a linear map $\alpha : A \rightarrow A$, the pair $\{ V_m^\alpha , \circ \}$ forms a graded pre-Lie algebra. Hence, the graded commutator 
%$$ [f, g] = f \circ g - (-1)^{mn} g \circ f, ~~ \text{ for } f \in V_m, g \in V_n$$
defines a graded Lie algebra structure on $\oplus V_m^\alpha$. 
More precisely, the graded Lie bracket $[-,-] : V^\alpha_m \times V^\alpha_n \rightarrow V^\alpha_{m+n}$ is given by
$$[f,g] = f \circ g - (-1)^{mn} g \circ f ,~~ \text{ for } f \in V^\alpha_m, g \in V^\alpha_n,$$
where $f \circ g = \sum_{i=0}^{m} (-1)^{ni} f \circ_i g$.
This gives an alternative and less computational proof of \cite[Theorem 3.3]{amm-ej-makh}.

Let $(A, \mu, \alpha)$ be a hom-associative algebra and consider the graded Lie algebra structure on $ \oplus V^\alpha_m$ defined above. Since $ C^{m}_\alpha (A, A) = V^\alpha_{m-1}$ for $m \geq 1$, the Hochschild type cochain groups $C^\bullet_\alpha (A, A)$ carries a shifted graded Lie algebra structure. Explicitly, the shifted graded Lie bracket
$$ [-,-]_\alpha : C^m_\alpha (A, A) \times C^n_\alpha (A, A) \rightarrow C^{m+n-1}_\alpha (A, A), ~~ m, n \geq 1$$ is given by
$$ [f, g]_\alpha = f \circ g - (-1)^{(m-1)(n-1)} g \circ f,~~~ \text{ for } f \in C^m_\alpha (A, A),~ g \in C^n_\alpha (A, A),$$
where $(f \circ g) (a_1, \ldots, a_{m+n-1}) = 
%\sum_{i=1}^m (-1)^{(n-1)(i-1)}  (f \circ_i g)
\sum_{i=1}^{m} (-1)^{(n-1)(i-1)} f (\alpha^{n-1} a_1, \ldots, g (a_i, \ldots, a_{i+n-1}), \ldots, \alpha^{n-1} a_{m+n-1} ).$

For any $f \in C^n_\alpha (A, A) = V_{n-1}^\alpha$, we observe that
\begin{align*}
&~( f \circ \mu - (-1)^{n-1} \mu \circ f) (a_1, \ldots, a_{n+1})\\
=&~ \sum_{i=1}^{n} (-1)^{i-1} f \big( \alpha(a_1), \ldots, \mu (a_i , a_{i+1}), \ldots, \alpha (a_{n+1}) \big) \\
&~ - (-1)^{n-1} \mu ( f (a_1, \ldots, a_n), \alpha^{n-1} a_{n+1} ) - (-1)^{n-1} (-1)^{n-1} \mu (\alpha^{n-1}a_1 , f (a_2, \ldots, a_{n+1})) \\
=&~  - \sum_{i=1}^{n} (-1)^{i} f \big( \alpha(a_1), \ldots, \mu (a_i , a_{i+1}), \ldots, \alpha (a_{n+1}) \big) \\
&~- (-1)^{n-1}  \mu (f (a_1, \ldots, a_n) , \alpha^{n-1} a_{n+1} )- \mu ( \alpha^{n-1} a_1 , f (a_2, \ldots, a_{n+1}) ) = - (\delta_\alpha f) (a_1, \ldots, a_{n+1}).
\end{align*}
Hence,
\begin{align}\label{diff-formula-brckt}
\delta_\alpha f = - ( f \circ \mu - (-1)^{n-1} \mu \circ f) = - [f , \mu]_\alpha = (-1)^{n-1} [\mu, f]_\alpha, ~\text{ for } f \in C^n_\alpha (A, A).
\end{align}

\begin{prop}\label{brack-leibniz}
	For any $f \in C^m_\alpha (A, A)$ and $g \in C^n_\alpha (A, A)$, $m, n \geq 1$, we have
$$\delta_\alpha  [ f, g]_\alpha = (-1)^{n+1} [\delta_\alpha f , g]_\alpha + [f, \delta_\alpha g]_\alpha.$$
\end{prop}
\begin{proof}
	It follows from (\ref{diff-formula-brckt}) that
	\begin{align*}
	\delta_\alpha [f, g]_\alpha =&~ (-1)^{m+n} [  \mu, [f, g]_\alpha]_\alpha \\
	=&~ (-1)^{m + n} [ [\mu, f]_\alpha, g]_\alpha + (-1)^{n+1} [ f, [\mu, g]_\alpha  ]_\alpha  \qquad \text{(shifted graded Jacobi identity)}\\
	=&~ (-1)^{n+1} [ \delta_\alpha f, g ]_\alpha + [f, \delta_\alpha g]_\alpha.
	\end{align*}
\end{proof}

\begin{remark}\label{rem-bracket}
It follows from Proposition \ref{brack-leibniz} that if $f \in Z^m_\alpha (A, A)$ and $g \in Z^n_\alpha (A, A)$, $m, n \geq 1$, then $[f , g ]_\alpha \in Z^{m+n-1}_\alpha (A, A).$ Moreover, if either $f \in B^m_\alpha (A, A)$ or $g \in B^n_\alpha (A, A)$ then $[f , g ]_\alpha \in B^{m+n-1}_\alpha (A, A).$ Therefore, the degree $-1$ graded Lie bracket $[-,-]_\alpha$ on $C^\bullet_\alpha (A,A)$ induces a degree $-1$ graded Lie bracket on $H^\bullet_\alpha (A, A)$. We denote the induced bracket on  $H^\bullet_\alpha (A, A)$ by the same symbol $[-,-]_\alpha$.
\end{remark}

\section{Cup product}
The aim of this section is to define a cup product on the cohomology of a hom-associative algebra. Let $(A, \mu, \alpha)$ be a hom-associative algebra. For $f \in C^m_\alpha (A,A)$ and $g \in C^n_\alpha (A, A)$ with $m , n \geq 1$, define $f \cup_\alpha g \in C^{m+n}_\alpha (A,A)$ by
\begin{align*}
 (f \cup_\alpha g) (a_1, \ldots, a_{m+n}) =&~ \mu ( f (\alpha^{n-1} a_1, \ldots, \alpha^{n-1} a_m) , g (\alpha^{m-1} a_{m+1}, \ldots, \alpha^{m-1} a_{m+n}) ) \\
 =&~ \mu (  \alpha^{n-1} f (a_1, \ldots, a_m), \alpha^{m-1} g (a_{m+1}, \ldots, a_{m+n})).
\end{align*}
(Note the powers of $\alpha$ in the above expression.) In the case of associative algebra (i.e. $\alpha =$ id), one gets the standard cup product (\ref{std-cup}) defined on the classical Hochschild cochains. It follows from the hom-associativity condition (\ref{hom-ass-cond}) that the cup product $\cup_\alpha$ defined on $C^\bullet_\alpha (A,A)$ is associative. In other words,
\begin{align}\label{cup-ass}
f \cup_\alpha (g \cup_\alpha h) = (f \cup_\alpha g) \cup_\alpha h,
\end{align}
for any $f \in C^m_\alpha (A, A) ~, g \in C^n_\alpha (A, A)$ and $h \in C^p_\alpha (A, A)$.

\begin{remark}\label{cup-circ}
Let $f \in C^m_\alpha (A, A)$ and $g \in C^n_\alpha (A, A)$ with $m, n \geq 1$. Consider them as elements of $V^\alpha_{m-1}$ and $V^\alpha_{n-1}$, respectively. We observe that the cup product $\cup_\alpha$ and the $\circ_i$ products defined in the beginning of Section \ref{sec-3} are related in the following way
\begin{align*}
 f \cup_\alpha g = ((\mu \circ_0 f) \circ_{m} g ) ~~ \text{ and } g \cup_\alpha f = ((\mu \circ_1 f) \circ_0 g).
 \end{align*}
\end{remark}

In the next result, we prove that the differential $\delta_\alpha$ on $C^\bullet_\alpha (A, A)$ satisfies the graded Leibniz rule with respect to the cup product.
\begin{prop}\label{cup-leibniz}
For $f \in C^m_\alpha (A,A)$ and $g \in C^n_\alpha (A, A)$ with $m , n \geq 1$, we have
$$ \delta_\alpha (f \cup_\alpha g) =  \delta_\alpha f \cup_\alpha g + (-1)^m f \cup_\alpha \delta_\alpha g.$$
\end{prop}

\begin{proof}
	For any $a_1, a_2 , \ldots, a_{m+n+1} \in A$, we have
\begin{align*}
&~(\delta_\alpha (f \cup_\alpha g)) (a_1, a_2, \ldots, a_{m+n+1}) \\
=&~ \mu (\alpha^{m+n-1} (a_1) , (f \cup_\alpha g)(a_2, \ldots, a_{m+n+1}) ) \\
&~ + \sum_{i=1}^{m+n} (-1)^i (f \cup_\alpha g) ( \alpha (a_1), \ldots, \alpha(a_{i-1}), \mu (a_i , a_{i+1}), \alpha (a_{i+2}), \ldots, \alpha (a_{m+n+1})  ) \\
&~ + (-1)^{m+n+1} \mu ((f \cup_\alpha g) (a_1, \ldots, a_{m+n}) , \alpha^{m+n-1} (a_{m+n+1}) ) \\
=&~ \mu \big(\alpha^{m+n-1} (a_1) , ~ \mu (  f(\alpha^{n-1} a_2, \ldots, \alpha^{n-1} a_{m+1}) , g (\alpha^{m-1} a_{m+2}, \ldots, \alpha^{m-1} a_{m+n+1})  ) \big) \\
&~ + \sum_{i=1}^{m} (-1)^i \mu ( f ( \alpha^n a_1, \ldots, \alpha^{n-1} \mu (a_i , a_{i+1}), \ldots, \alpha^n a_{m+1} ) ,~ g (\alpha^m a_{m+2}, \ldots, \alpha^m a_{m+n+1})) \\
&~ + \sum_{i=1}^{n} (-1)^{m+i} \mu ( f ( \alpha^n a_1, \ldots, \alpha^n a_m   ) ,~  g ( \alpha^m a_{m+1}, \ldots, \alpha^{m-1} \mu (a_{m+i} , a_{m+i+1}), \ldots, \alpha^{m} (a_{m+n+1})   )) \\
&~ + (-1)^{m+n+1} \mu \big( ~\mu \big( f ( \alpha^{n-1} a_1, \ldots, \alpha^{n-1} a_m  )  ,  g ( \alpha^{m-1} a_{m+1} , \ldots, \alpha^{m-1} a_{m+n} ) \big) ,~ \alpha^{m+n-1} (a_{m+n+1}) \big) \\
=&~ \mu \big(\bigg[  \mu (~ \alpha^{m+n-2} (a_1) , f (\alpha^{n-1} a_2, \ldots, \alpha^{n-1} a_{m+1}) ) \\
&~+  \sum_{i=1}^{m} (-1)^i  f \big(  \alpha (\alpha^{n-1} a_1) , \ldots, \mu (\alpha^{n-1} a_i , \alpha^{n-1} a_{i+1}), \ldots, \alpha (\alpha^{n-1} a_{m+1}) \big) \\
&~+ (-1)^{m+1} \mu (~ f (\alpha^{n-1} a_1, \ldots, \alpha^{n-1} a_m) ,~ \alpha^{m+n-2} a_{m+1})   \bigg] , g (\alpha^m a_{m+2}, \ldots, \alpha^m a_{m+n+1}) \big) \\
&~ + (-1)^m \mu \big(~ f (  \alpha^n a_1, \ldots, \alpha^n a_m  ) ,  \bigg[   \mu (~ \alpha^{m+n-2} a_{m+1} ,~ g (a^{m-1} a_{m+2}, \ldots, \alpha^{m-1} a_{m+n+1})) \\
&~ + \sum_{i=1}^{n} (-1)^i g \big(  \alpha (\alpha^{m-1} a_{m+1}), \ldots, \mu (\alpha^{m-1} a_{m+i} , \alpha^{m-1} a_{m+i+1}), \ldots, \alpha ( \alpha^{m-1} a_{m+n+1})  \big)   \\
&~ + (-1)^{n+1}  \mu \big(~ g ( \alpha^{m-1} a_{m+1}, \ldots, \alpha^{m-1} a_{m+n}  ) , ~ \alpha^{m+n-2} (a_{m+n+1}) \big) \bigg] \big)  \quad (\text{using hom-associativity})\\
=&~ \mu \big( (\delta_\alpha f) (\alpha^{n-1} a_1, \ldots, \alpha^{n-1} a_{m+1}) ,~ g (\alpha^m a_{m+2}, \ldots, \alpha^m a_{m+n+1}) \big) \\
&~ + (-1)^m  \mu \big( f ( \alpha^n a_1, \ldots, \alpha^n a_m ) , ~(\delta_\alpha g) ( \alpha^{m-1} a_{m+1}, \ldots, \alpha^{m-1} a_{m+n+1} ) \big) \\
=&~ \big[  \delta_\alpha f \cup_\alpha g  + (-1)^m f \cup_\alpha \delta_\alpha g \big] (a_1, a_2, \ldots, a_{m+n+1}).
\end{align*}
Hence the result follows.
\end{proof}

\begin{remark}\label{rem-cup}
It follows from Proposition \ref{cup-leibniz} that if $f \in Z^m_\alpha (A, A)$ and $g \in Z^n_\alpha (A, A)$, $m, n \geq 1$, then $f \cup_\alpha g \in Z^{m+n}_\alpha (A, A).$ Moreover, if either $f \in B^m_\alpha (A, A)$ or $g \in B^n_\alpha (A, A)$ then $f \cup_\alpha g \in B^{m+n}_\alpha (A, A).$ Therefore, the cup product $\cup_\alpha$ on $C^\bullet_\alpha (A,A)$ induces an associative (cup) product on $H^\bullet_\alpha (A, A)$. We denote the induced cup product on  $H^\bullet_\alpha (A, A)$ by the same symbol $\cup_\alpha$. It follows from condition (\ref{cup-ass}) that $\cup_\alpha$ is associative on $H^\bullet_\alpha (A, A)$.
\end{remark}

The cup product $\cup_\alpha$ on the cochain level is not graded commutative.
In the next proposition, we will prove that the cup product $\cup_\alpha$ is graded commutative at the cohomology level. The proof is similar to the classical case \cite{gers}.

\begin{prop}\label{some-lemma-1}
Let $(A, \mu , \alpha)$ be a hom-associative algebra. For any $f \in C^m_\alpha (A, A)$ and $g \in C^n_\alpha (A, A)$, we have
$$ f \circ \delta_\alpha g - \delta_\alpha (f \circ g) + (-1)^{n-1} \delta_\alpha f \circ g = (-1)^{n-1} \big( g \cup_\alpha f - (-1)^{mn} f \cup_\alpha g \big).$$
\end{prop}

\begin{proof}
	By using (\ref{diff-formula-brckt}) we get
\begin{align}\label{grad-comm}
&~f \circ \delta_\alpha g - \delta_\alpha ( f \circ g) + (-1)^{n-1} \delta_\alpha f \circ g \nonumber \\
=&~ (-1)^{n-1} f \circ ( \mu \circ g) - f \circ (g \circ \mu) - (-1)^{m+n -2} \mu \circ ( f \circ g) + (f \circ g) \circ \mu \nonumber \\
&~ + (-1)^{n-1} (-1)^{m-1} (\mu \circ f) \circ g - (-1)^{n-1} (f \circ \mu) \circ g .
\end{align}
Since $\{ V^\alpha_m , \circ_i \}$ is a pre-Lie system and $f \in V^\alpha_{m-1}, ~ g \in V^\alpha_{n-1}, ~\mu \in V^\alpha_1$, by Theorem \ref{thm2-gers}(ii) we get
$$ (f \circ g) \circ \mu - f \circ (g \circ \mu ) = (-1)^{n-1} (f \circ \mu) \circ g - (-1)^{n-1} f \circ (\mu \circ g).$$
Substituting the above relation in (\ref{grad-comm}) we get
\begin{align*}
 &~f \circ \delta_\alpha g - \delta_\alpha ( f \circ g) + (-1)^{n-1} \delta_\alpha f \circ g \\
 =&~ (-1)^{m+n} [  (\mu \circ f) \circ g - \mu \circ ( f \circ g)    ] \\
 =&~ (-1)^{m+n} [ (-1)^{m (n-1)}  (\mu \circ_0 f) \circ_{m} g  + (-1)^{m-1} (\mu \circ_1 f) \circ_0 g  ] ~~\quad (\text{by Theorem \ref{thm2-gers} (i)})\\
 =&~ (-1)^{m+n} [ (-1)^{m (n-1)} f \cup_\alpha g + (-1)^{m-1} g \cup_\alpha f   ] ~~ \quad \text{(by Remark \ref{cup-circ})}\\
 =&~ (-1)^{mn + n} f \cup_\alpha g + (-1)^{n-1} g \cup_\alpha f  = (-1)^{n-1} [ g \cup_\alpha f - (-1)^{mn} f \cup_\alpha g ].
\end{align*}
\end{proof}

\begin{remark}\label{rem-cup-comm}
It follows from Proposition \ref{some-lemma-1} that if $f \in Z^m_\alpha (A, A)$ and $g \in Z^n_\alpha (A, A)$ then
$$    (-1)^n [ g \cup_\alpha f - (-1)^{mn} f \cup_\alpha g] =  \delta_\alpha (f \circ g).$$
Hence, at the level of cohomology, the cup product is graded commutative.
\end{remark}

\section{Leibniz rule}

Let $(A, \mu, \alpha)$ be a hom-associative algebra. In this section, we prove that the degree $-1$ graded Lie bracket $[-,-]_\alpha$ and the cup product $\cup_\alpha$ on the Hochschild type cohomology $H^\bullet_\alpha (A, A)$ satisfies the Leibniz rule of a Gerstenhaber algebra. The proof is similar to the classical case \cite{gers} involving $\alpha$.

For simplicity we will use the following notations. The multiplication $\mu : A \times A \rightarrow A,~ (a,b) \mapsto \mu (a,b)$ is denoted by the dot $\cdot$, therefore, $a \cdot b = \mu (a, b)$. Let $a_1, a_2, \ldots \in A$ be arbitrary. For any $i \leq j$ and $\theta \geq 0$, we will write $a_{ij}^\theta$ for $(\alpha^\theta a_i, \ldots, \alpha^\theta a_j)$ and $a^\theta_i$ for $\alpha^\theta a_i$. Sometimes we also omit the arguments from an expression of the form
$$ f (\alpha^\theta a_\lambda, \alpha^\theta a_{\lambda + 1}, \ldots, \alpha^\theta a_{\lambda + m -1}) ~~~ \text{  or  } ~~~  g(\alpha^\theta a_\lambda, \alpha^\theta a_{\lambda +1}, \ldots, \alpha^\theta a_{\lambda + n -1})$$
and simply write them as $f^\theta$ or $g^\theta$. It is easy to write down the correct arguments as it depend only on the first index.

Let $f \in C^m_\alpha (A, A), ~ g \in C^n_\alpha (A, A)$ and $h \in C^p_\alpha (A, A)$ be three cochains. Let  $1 \leq i \leq p-1$ and $ m+i \leq j \leq m+p-1$. For any $a_1, a_2, \ldots, a_{m+n+p-1} \in A$,  we set
\begin{align*}
h_{i,j} =&~ (\alpha^{m+n+p -3} a_1) \cdot h ( a_{2,i}^{m+n-2},~ f^{n-1},~ a_{i+m+1, j}^{m+n-2},~ g^{m-1},~ a_{j+n+1, m+n+p-1}^{m+n-2}) \\
&~ + \sum_{\lambda = 1}^{i-1} (-1)^\lambda ~ h ( a^{m+n-1}_{1,\lambda -1} , ~ \alpha^{m+n-2} (a_\lambda \cdot a_{\lambda+1}), ~a^{m+n-1}_{\lambda +2 , i}, ~ f^n, ~ a_{i+m +1, j}^{m+n-1}, ~ g^m, ~ a^{m+n-1}_{j+n+1, m+n+p-1}) \\
&~ + (-1)^i ~ h (  a^{m+n-1}_{1, i-1},~ a_i^{m+n-2} \cdot f^{n-1},~ a^{m+n-1}_{i+m+1, j}, ~ g^m, ~ a^{m+n-1}_{j+n+1, m+n+p-1}),
\end{align*}
\begin{align*}
h'_{i,j} =&~  (-1)^{m+i-1} ~ h (a^{m+n-1}_{1, i-1}, ~ f^{n-1} \cdot a^{m+n-2}_{i+m},~ a^{m+n-1}_{i+m+1, j}, ~ g^m, ~ a^{m+n-1}_{j+n+1, m+n+p-1}) \\
&~ + \sum_{\lambda = m+i}^{j-1} (-1)^\lambda~ h (a^{m+n-1}_{1, i-1}, ~f^n,~ a^{m+n-1}_{i+m, \lambda -1},~ \alpha^{m+n-2} (a_\lambda \cdot a_{\lambda +1}),~ a^{m+n-1}_{\lambda +2, j},~ g^m,~ a^{m+n-1}_{j+n+1, m+n+p-1}) \\
&~ + (-1)^j ~h ( a^{m+n-1}_{1, i-1}, ~ f^n, ~ a^{m+n-1}_{i+m, j-1},~ a_j^{m+n-2} \cdot g^{m-1}, ~ a_{j+n+1, m+n+p-1}^{m+n-1}) 
\end{align*}
and
\begin{align*}
h''_{i,j} =&~  (-1)^{j+n-1} ~ h ( a^{m+n-1}_{1, i-1}, ~ f^n, ~ a^{m+n-1}_{i+m, j-1},~ g^{m-1} \cdot a_{j+n}^{m+n-2}, ~ a_{j+n+1, m+n+p-1}^{m+n-1}  ) \\
&~ + \sum_{\lambda = j+n}^{m+n+p-2} (-1)^\lambda~ h ( a^{m+n-1}_{1, i-1}, ~ f^n, ~ a^{m+n-1}_{i+m, j-1}, ~ g^m,~ a^{m+n-1}_{j+n, \lambda -1}, \alpha^{m+n-2} (a_\lambda \cdot a_{\lambda +1}),~ a^{m+n-1}_{\lambda +2, m+n+p-1} ) \\
&~ + (-1)^{m+n+p-1} h ( a^{m+n-2}_{i, i-1},~ f^{n-1},~ a^{m+n-2}_{i+m, j-1},~ g^{m-1},~ a^{m+n-2}_{j+n, m+n+p-2}) \cdot \alpha^{m+n+p-3} (a_{m+n+p-1}).
\end{align*}

We need the following useful result.
\begin{lemma}\label{h-new-lemma}
Let $(A, \mu, \alpha)$ be a hom-associative algebra and $f \in Z^m_\alpha (A, A),~ g \in Z^n_\alpha (A, A), ~ h \in C^p_\alpha (A, A)$. For $1 \leq i \leq p-1$ and $m+i \leq j \leq m+p-1$,
$$h_{i,j} + h'_{i,j} + h''_{i,j} = \delta_\alpha ((h \circ_{i-1}f) \circ_{j-1} g) (a_1, a_2, \ldots, a_{m+n+p-1}).$$
\end{lemma}

\begin{proof}
We have
\begin{align*}
&~\delta_\alpha ((h \circ_{i-1}f) \circ_{j-1} g) (a_1, a_2, \ldots, a_{m+n+p-1}) \\
=&~  (\alpha^{m+n+p -3} a_1) \cdot ((h \circ_{i-1}f) \circ_{j-1} g) ( a_2, \ldots, a_{m+n+p-1}) \\
&~ + \sum_{\lambda = 1}^{m+n+p-2} (-1)^\lambda ((h \circ_{i-1}f) \circ_{j-1} g) ( \alpha (a_1), \ldots, \alpha (a_{\lambda -1}), a_\lambda \cdot a_{\lambda +1}, \alpha (a_{\lambda +2}), \ldots, \alpha (a_{m+n+p-1})     ) \\
&~ + (-1)^{m+n+p-1} ((h \circ_{i-1}f) \circ_{j-1} g) (a_1, a_2, \ldots, a_{m+n+p-2}) \cdot \alpha^{m+n+p-3} (a_{m+n+p-1}) \\
=&~  (\alpha^{m+n+p -3} a_1) \cdot h ( a_{2,i}^{m+n-2},~ f^{n-1},~ a_{i+m+1, j}^{m+n-2},~ g^{m-1},~ a_{j+n+1, m+n+p-1}^{m+n-2})\\
&~ + \sum_{\lambda = 1}^{i-1} + \sum_{\lambda = i}^{m+i-1} + \sum_{\lambda = m+i}^{j-1} + \sum_{\lambda = j}^{j+n -1} + \sum_{\lambda = j+n}^{m+n + p-2} \\
&~ + (-1)^{m+n+p-1} h ( a^{m+n-2}_{i, i-1},~ f^{n-1},~ a^{m+n-2}_{i+m, j-1},~ g^{m-1},~ a^{m+n-2}_{j+n, m+n+p-2}) \cdot \alpha^{m+n+p-3} (a_{m+n+p-1}).
\end{align*} 
Note that
\begin{align*}
\sum_{\lambda = 1}^{i-1} = \sum_{\lambda = 1}^{i-1} (-1)^\lambda ~ h ( a^{m+n-1}_{1,\lambda -1} , ~ \alpha^{m+n-2} (a_\lambda \cdot a_{\lambda+1}), ~a^{m+n-1}_{\lambda +2 , i}, ~ f^n, ~ a_{i+m +1, j}^{m+n-1}, ~ g^m, ~ a^{m+n-1}_{j+n+1, m+n+p-1}),
\end{align*}
\begin{align*}
 \sum_{\lambda = m+i}^{j-1} =  \sum_{\lambda = m+i}^{j-1} (-1)^\lambda~ h (a^{m+n-1}_{1, i-1}, ~f^n,~ a^{m+n-1}_{i+m, \lambda -1},~ \alpha^{m+n-2} (a_\lambda \cdot a_{\lambda +1}),~ a^{m+n-1}_{\lambda +2, j},~ g^m,~ a^{m+n-1}_{j+n+1, m+n+p-1})
\end{align*}
and
\begin{align*}
\sum_{\lambda = j+n}^{m+n+p-2} = \sum_{\lambda = j+n}^{m+n+p-2} (-1)^\lambda~ h ( a^{m+n-1}_{1, i-1}, ~ f^n, ~ a^{m+n-1}_{i+m, j-1}, ~ g^m,~ a^{m+n-1}_{j+n, \lambda -1},~ \alpha^{m+n-2} (a_\lambda \cdot a_{\lambda +1}),~ a^{m+n-1}_{\lambda +2, m+n+p-1} ).
\end{align*}
Moreover, since $f$ is a cocycle, we observe that
\begin{align*}
\sum_{\lambda = i}^{m + i-1} = &~(-1)^i ~ h (  a^{m+n-1}_{1, i-1},~ a_i^{m+n-2} \cdot f^{n-1},~ a^{m+n-1}_{i+m+1, j}, ~ g^m, ~ a^{m+n-1}_{j+n+1, m+n+p-1}) \\
&~ + (-1)^{m+i-1} ~ h (a^{m+n-1}_{1, i-1}, ~ f^{n-1} \cdot a^{m+n-2}_{i+m},~ a^{m+n-1}_{i+m+1, j}, ~ g^m, ~ a^{m+n-1}_{j+n+1, m+n+p-1}).
\end{align*}
Similarly, $g$ is a cocycle implies that
\begin{align*}
\sum_{\lambda = j}^{j + n -1} = &~ (-1)^j ~h ( a^{m+n-1}_{1, i-1}, ~ f^n, ~ a^{m+n-1}_{i+m, j-1},~ a_j^{m+n-2} \cdot g^{m-1}, ~ a_{j+n+1, m+n+p-1}^{m+n-1}) \\
&~ + (-1)^{j+n-1} ~ h ( a^{m+n-1}_{1, i-1}, ~ f^n, ~ a^{m+n-1}_{i+m, j-1},~ g^{m-1} \cdot a_{j+n}^{m+n-2}, ~ a_{j+n+1, m+n+p-1}^{m+n-1}  ).
\end{align*}
Hence, the result follows from the definition of $h_{i,j}, ~ h'_{i,j}$ and $h''_{i,j}$.
\end{proof}

The above Lemma will be used to prove the next proposition. We first take $f \in Z^m_\alpha (A, A), ~ g \in Z^n_\alpha (A,A)$ and $h \in Z^p_\alpha (A, A)$, for $m, n, p \geq 1$. Then one can easily check that
$$h \circ (f \cup_\alpha g) - (-1)^{n (p-1)} (h \circ f) \cup_\alpha g  - f \cup_\alpha (h \circ g) = 0, ~~ \text{ whenever } p= 1.$$
To prove this, one only needs that $h$ is a $1$-cocycle.
If $p \neq 1$, the left hand side of the above equality may not be zero. However, the next proposition suggests that this is given by a coboundary.
This type of verifications appear in many places \cite[Theorem 5]{gers} (see also \cite{trad}).

\begin{prop}\label{h-lemma}
Let $(A, \mu, \alpha)$ be a hom-associative algebra and $f \in Z^m_\alpha (A, A), ~ g \in Z^n_\alpha (A, A)$ and $h \in Z^p_\alpha (A, A)$. Let
$$ H = \sum_{i=0}^{p-2} \sum_{j= m+i}^{ m + p - 2} (-1)^{(m-1)i + (n-1)j} (h \circ_i f) \circ_j g.$$
Then $H \in C^{m+n + p -2}_\alpha (A, A)$ and
$$ \delta_\alpha H = (-1)^{(m-1)n} \big[ h \circ (f \cup_\alpha g) - (-1)^{n (p-1)} (h \circ f) \cup_\alpha g  - f \cup_\alpha (h \circ g)  \big].$$
\end{prop}
\begin{proof}
For $1 \leq i \leq p-1, ~ m+i \leq j \leq m + p-1$ and $a_1, a_2, \ldots, a_{m+n+p-1} \in A,$ we have from Lemma \ref{h-new-lemma} that
\begin{align}\label{new-four}
 h_{i,j} + h'_{i,j} + h''_{i,j} = \delta_\alpha (  (h \circ_{i-1} f) \circ_{j-1} g) (a_1, a_2, \ldots, a_{m+n+p-1}).
\end{align}
In particular, if $m +i + 1 \leq j \leq m + p -2$, then
\begin{align}\label{label-ts}
&~ h_{i,j}  + (-1)^{m-1} h'_{i+1, j} + (-1)^{(m-1) + (n - 1)} h''_{i+1, j+1} \\
=&~  (\delta_\alpha h)(  a^{m+n-2}_{1, i}, ~ f^{n-1}, ~ a^{m+n-2}_{i+m+1, j}, ~ g^{m-1},~  a^{m+n-2}_{j+n+1, m+n+p -1}) = 0 \nonumber,
\end{align}
as $h$ is a cocycle. One may also extend the range of indices for which 
(\ref{label-ts}) is valid by setting
\begin{align}
h_{0,j} =&~ (f \cup_\alpha (h \circ_{j - m} g))(a_1, \ldots, a_{m+n+p-1}) \quad \text{ for } j = m, m+1, \ldots, m+p-1 \label{equan1}\\
h'_{i, m+i-1} =&~ (-1)^{m +i -1} (h \circ_{i-1} (f \cup_\alpha g))(a_1, \ldots, a_{m+n+p-1}) \quad \text{ for } i = 1, 2, \ldots, p  \label{equan2}\\
h''_{i, m+p} =&~ (-1)^{m+n+p-1} ((h \circ_{i-1} f) \cup_\alpha g)(a_1, \ldots, a_{m+n+p-1}) \quad \text{for } i = 1, 2, \ldots, p. \label{equan3}
\end{align}
Therefore,  (\ref{label-ts}) hold for all those pairs $(i,j)$ such that $0 \leq i \leq p-1$ and $m+i \leq j \leq m + p -1$. Hence,
\begin{align}\label{eqn-th}
\sum_{(i,j)}^{} (-1)^{(m-1)(i-1) + (n-1)(j-1)} ~  \big[ h_{i,j}  + (-1)^{m-1} h'_{i+1, j} + (-1)^{(m-1) + (n - 1)} h''_{i+1, j+1} \big]  = 0 ,
\end{align}
where the sum is indexed over all the pairs $(i,j)$ such that $0 \leq i \leq p-1$ and $m+i \leq j \leq m + p -1$. Moreover, it follows from (\ref{new-four}) that
$$ (\delta_\alpha H)(a_1, \ldots, a_{m+n+p-1}) = \sum_{i=1}^{p-1} \sum_{j=m+i}^{m+p-1} (-1)^{(m-1)(i-1) + (n-1)(j-1)} ( h_{i,j} + h'_{i,j} + h''_{i,j}  ).$$
One can easily observe that the whole expression of $(\delta_\alpha H)(a_1, \ldots, a_{m+n+p-1})$ is contained in the left hand side of (\ref{eqn-th}). Although, the left hand side of (\ref{eqn-th}) contain additional terms (with some signs) which are of the forms (\ref{equan1}), (\ref{equan2}) and (\ref{equan3}). More precisely, by substituting the expression of 
$(\delta_\alpha H)(a_1, \ldots, a_{m+n+p-1})$ in (\ref{eqn-th}), we get
\begin{align*}
(\delta_\alpha H)(a_1, \ldots, a_{m+n+p-1})  &~+ \sum_{j= m}^{m + p -1} (-1)^{-(m-1) 
+ (n-1)(j-1)} h_{0,j} \\
&~+ \sum_{i=0}^{p-1} (-1)^{(m-1)(i-1) + (n-1)(m+i-1) + (m-1)} h'_{i+1, m+i}  \\
&~ +  \sum_{i=0}^{p-1} (-1)^{(m-1)(i-1) + (n-1)(m+p-2) + (m-1) + (n-1)} h''_{i+1, m+p} = 0.
\end{align*}
Then by a straightforward calculation using (\ref{equan1}), (\ref{equan2}) and (\ref{equan3}), we get
$$ \delta_\alpha H + (-1)^{(m-1) n} f \cup_\alpha (h \circ g) + (-1)^{(m-1)n +1} h \circ (f \cup_\alpha g) + (-1)^{(m-1)n + n (p-1)} (h \circ f) \cup_\alpha g = 0.$$
Hence the result follows.
\end{proof}

\begin{prop}\label{h-lemma2}
	For any $f \in C^m_\alpha (A, A),~ g \in C^n_\alpha (A, A)$ and $h \in C^p_\alpha (A, A)$, we have
 $$ (f \cup_\alpha g) \circ h = (f \circ h ) \cup_\alpha g + (-1)^{m (p-1)} f \cup_\alpha ( g \circ h).$$
\end{prop}

\begin{proof}
For any $a_1, a_2, \ldots, a_{m+n+p-1} \in A$,
\begin{align*}
&~((f \cup_\alpha g) \circ h) (a_1, a_2, \ldots, a_{m+n+p-1}) \\
=&~ \sum_{i=1}^{m+n} (-1)^{(p-1)(i-1)} ~ (f \cup_\alpha g) (  a^{p-1}_{1, i-1}, h (a_i, \ldots, a_{i+p-1}), a^{p-1}_{i+p, m+n+p-1}) \\
=&~ \sum_{i=1}^{m} (-1)^{(p-1)(i-1)} \big(\alpha^{n-1} f ( \alpha^{p-1} a_1, \ldots, h (a_i, \ldots, a_{i+p-1}), \ldots, \alpha^{p-1} a_{m+p-1}) \big) \cdot\\ &~ \qquad \qquad \big( \alpha^{m+p-2} g (a_{m+p}, \ldots, a_{m+n+p-1})\big) \\
&~ + \sum_{i=m+1}^{m+n} (-1)^{(p-1)(i-1)} ~ \big( \alpha^{n+p-2} f(a_1, \ldots, a_m) \big) \cdot \\
&~ \qquad \qquad \big( \alpha^{m-1} g(\alpha^{p-1} a_{m+1}, \ldots, h(a_i, \ldots, a_{i+p-1}), \ldots, \alpha^{p-1} a_{m+n+p-1}) \big) \\
=&~ \big( \sum_{i=1}^{m} (-1)^{(p-1)(i-1)} \alpha^{n-1} f ( \alpha^{p-1} a_1, \ldots, h (a_i, \ldots, a_{i+p-1}), \ldots, \alpha^{p-1} a_{m+p-1}) \big) \cdot \\ &~ \qquad \qquad \big( \alpha^{m+p-2} g (a_{m+p}, \ldots, a_{m+n+p-1})\big) \\
 &~+ (-1)^{m(p-1)} \big( \alpha^{n+p-2} f(a_1, \ldots, a_m) \big) \cdot \\
 &~ \qquad \qquad  \big(  \sum_{i=1}^{n} (-1)^{(p-1)(i-1)} \alpha^{m-1} g (\alpha^{p-1} a_{m+1}, \ldots, h(a_{m+i}, \ldots, a_{m+p-1}), \ldots, \alpha^{p-1} a_{m+n+p-1}))  \big)\\
 =&~ \big(\alpha^{n-1} (f \circ h)(a_1, \ldots, a_{m+p-1}) \big) \cdot \big( \alpha^{m+p-2} g (a_{m+p}, \ldots, a_{m+n+p-1}) \big) \\
 &~ + (-1)^{m(p-1)}  \big( \alpha^{n+p-2} f(a_1, \ldots, a_m)  \big) \cdot \big(  \alpha^{m-1} (g \circ h) (a_{m+1}, \ldots, a_{m+n+p-1})  \big) \\
 =&~ \big(  (f \circ h ) \cup_\alpha g + (-1)^{m (p-1)} f \cup_\alpha ( g \circ h)  \big) (a_1, a_2, \ldots, a_{m+n+p-1}).
\end{align*}
\end{proof}

Finally, we prove the following.
\begin{prop}\label{final-prop}
	If $f \in Z^m_\alpha (A, A), ~g \in Z^n_\alpha (A, A) $ and $h \in Z_\alpha^p (A, A)$ then
	$$ [f \cup_\alpha g , h]_\alpha - [f, h]_\alpha \cup_\alpha g - (-1)^{m (p-1)} f \cup_\alpha [g, h]_\alpha = ~ \pm \delta_\alpha H.$$
\end{prop}
\begin{proof}
	We have
\begin{align*}
&~[f \cup_\alpha g , h]_\alpha - [f, h]_\alpha \cup_\alpha g - (-1)^{m (p-1)} f \cup_\alpha [g, h]_\alpha \\
=&~ ( f \cup_\alpha g) \circ h - (-1)^{ (p-1) (m+n-1)  } h \circ ( f \cup_\alpha g) - (f \circ h) \cup_\alpha g + (-1)^{(m-1)(p-1)} (h \circ f) \cup_\alpha g \\
&~ - (-1)^{m (p-1)} f \cup_\alpha (g \circ h) + (-1)^{m (p-1)} (-1)^{(n-1)(p-1)} f \cup_\alpha (h \circ g) \\
=&~ - (-1)^{m (p-1)} \big( ~  (-1)^{(n-1)(p-1)} h \circ (f \cup_\alpha g) + (-1)^p (h \circ f) \cup_\alpha g - (-1)^{(n-1)(p-1)} f \cup_\alpha (h \circ g) ~ \big) \\
&~ \hspace*{3.5in} (\text{by Proposition \ref{h-lemma2}}) \\
=&~ \pm \big( ~  h \circ (f \cup_\alpha g)  -  f \cup_\alpha (h \circ g) + (-1)^{np - n + 1}  (h \circ f) \cup_\alpha g ~\big) \\
=&~ \pm \delta_\alpha H. ~~ \quad \text{(by Proposition \ref{h-lemma})}
\end{align*}
\end{proof}

\begin{remark}\label{rem-leibniz}
It follows from Proposition \ref{final-prop} that at the level of cohomology, the degree $-1$ graded Lie bracket $[-,-]_\alpha$ satisfies the Leibniz rule with respect to the cup product $\cup_\alpha$ on the cohomology.
\end{remark}

%Thus, we conclude that the Hochschild cohomology $\oplus_{n \geq 2} H^n_\alpha (A, A)$ of a hom-associative algebra $(A, \mu, \alpha)$ carries a Gerstenhaber algebra structure with respect to the degree $-1$ graded Lie bracket $[-,-]_\alpha$ and the cup product $\cup_\alpha$. 

As a summary, we get the following.
\begin{thm}\label{final-thm-old}
Let $(A, \mu, \alpha)$ be a hom-associative algebra. Then its Hochschild type cohomology $H^\bullet_\alpha (A, A)$ inherits a Gerstenhaber algebra structure.
\end{thm}

When $\alpha =$ identity, one can extend the bracket and the cup product from the zero-th cochains and compatible to the Hochschild coboundary as Propositions \ref{brack-leibniz}, \ref{cup-leibniz}. Thus, one gets the Gerstenhaber algebra structure on Hochschild cohomology $\oplus_{n \geq 0} H^n (A, A)$ \cite{gers}.

\begin{remark}
It was shown in \cite{das1} that the Hochschild cochain complex defining the cohomology of a hom-associative algebra inherits a structure of an operad with a multiplication.  As a consequence, we get a Gerstenhaber algebra structure on the cohomology. However, the proof of the present paper is direct and therefore it passes through some important propositions which might be useful as well. Moreover, the results of the present paper (arXiv:1805.01207) was obtained by the author before the paper \cite{das1} was submitted in arXiv.
\end{remark}

\begin{remark}
Let $(A, \mu)$ be an associative algebra and $\alpha : A \rightarrow A$ be an associative algebra morphism. We consider the hom-associative algebra $(A, \alpha \circ \mu, \alpha)$ and the corresponding Hochschild cochain complex are given by $(C^\bullet_\alpha (A, A), \delta_\alpha)$. We have seen that these cochain groups have a cup product and a degree $-1$ graded Lie bracket which makes the corresponding cohomology $H^\bullet_\alpha (A, A)$ a Gerstenhaber algebra.

On the other hand, we may also consider the classical Hochschild complex $(C^\bullet (A, A), \delta)$ of the associative algebra $A$. Since $\alpha : A \rightarrow A$ is an associative algebra morphism, the Hochschild coboundary $\delta$ restricts to $C^\bullet_\alpha (A, A)$. In other words, $(C^\bullet_\alpha (A, A), \delta)$ is a subcomplex of the Hochschild complex  $(C^\bullet (A, A), \delta)$. Moreover, the cup product and the degree $-1$ graded Lie bracket on the Hochschild cochains $C^\bullet (A, A)$ restricts to $C^\bullet_\alpha (A, A)$. Hence, we get a Gerstenhaber subalgebra of $H^\bullet (A, A)$. However, this Gerstenhaber subalgebra is different than the Gerstenhaber algebra $H^\bullet_\alpha (A, A)$ induced from the hom-associative algebra $(A, \alpha \circ \mu, \alpha)$.
\end{remark}

\medskip

\noindent {\bf Acknowledgements.} The author would like to thank the referee for his/her valuable comments on the earlier version of the manuscript. The research has been carried out when the author was a research fellow at Indian Statistical Institute, Kolkata (India). He wishes to thank the Institute for their support.

\end{document}